\documentclass[11pt]{amsart}

\usepackage{latexsym}

\usepackage{amsmath,amssymb,amsthm,amsfonts,latexsym}
\usepackage{pstricks, pst-node, pst-text, pst-3d}
\usepackage[bookmarks]{hyperref}
\hypersetup{backref, colorlinks=true}
\let\=\partial

\newtheorem {pro}{Proposition}[section]
\newtheorem {thm}[pro]{Theorem}%[section]
\newtheorem{lem}[pro]{Lemma}

\theoremstyle{definition}
 \newtheorem {rem}[pro]{Remark}%[section]
 
\newtheorem {dfn}[pro]{Definition}%[section]
\newtheorem {defs}[pro]{Definitions}
\newtheorem {exa}[pro]{Example}
\newcommand{\xt}{\tilde{X}}

\newcommand{\R} {\mathbb{R}}

\newcommand{\p}{\check{p}}

\newcommand{\ep}{\varepsilon}
\newcommand{\pa}{\partial}

\title {{\bf A  Lefschetz duality for intersection homology}}



\author{Guillaume Valette}

\address
{Instytut Matematyczny PAN, ul. \'Sw. Tomasza 30, 31-027 Krak\'ow,
Poland} \email{gvalette@impan.pl}


\thanks{}


\begin{document}

\maketitle

\begin{abstract}
We prove a  Lefschetz duality result  for intersection homology. Usually, this result applies to pseudomanifolds with boundary which are assumed to have a "collared neighborhood of their boundary". Our duality does not need this assumption and is a generalization of the classical one.
\end{abstract}

\section{Introduction}
The main feature of intersection homology is that it satisfies Poincar\'e duality for a large class of singular sets, called pseudomanifolds.  This duality is particularly nice when the considered singular sets may be stratified by a stratification having only even dimensional strata, like for instance the complex analytic sets.

In their fundamental paper M. Goresky and R.  MacPherson \cite{gm1} (see also \cite{gm2}) introduced intersection homology, showed that it is finitely generated and independent of the stratification and established their generalized Poincar\'e duality. They also introduce the notion of pseudomanifold with boundary to which a generalized Lefschetz duality applies.

A pseudomanifold is a subset $X$ for which the singular locus is of codimension at least $2$ in $X$ (and is nowhere dense in $X$). Pseudomanifolds with boundary are couples $(X;\pa X)$ such that $X\setminus \pa X$ and $\pa X$ are pseudomanifolds and such that $\pa X$ has a neighborhood  in $X$ which is homeomorphic to a product $\pa X \times [0;1]$. In this paper, we show how the last requirement can be left out without affecting Lefschetz duality.

We consider couples $(X;\partial X)$ with $X$ manifold with boundary $\pa X$ near the top stratum of $\pa X$, such that  $X\setminus \pa X$ and $\partial X$ are both stratified pseudomanifolds  that we call stratified $\pa$-pseudomanifolds, and establish a  more general version of Lefschetz duality. % It is also worthy of notice that we also have the exact sequence for the pair $(X;\pa X)$.

This approach is different from the one developed by G. Friedman in \cite{f1,f2,f3} where the author  obtained several very interesting results on pseudomanifolds with possibly one codimensional strata with generalized perversities. The novelty of the present paper is that the allowable chains of $\pa X$ are allowable in $X$.  

In \cite{vlinfty}, the author proves that the cohomology of  $L^\infty$ forms on a compact subanalytic pseudomanifold is isomorphic to intersection cohomology in the maximal perversity.  In \cite{vl1}, we give a Lefschetz duality theorem, relating the $L^\infty$ cohomology to the so-called Dirichlet $L^1$-cohomology.  As a corollary of these two results, on compact subanalytic pseudomanifolds,  we got that the Dirichlet  $L^1$ cohomology is isomorphic to intersection cohomology in the zero perversity. The Lefschetz duality of \cite{vl1} is true for any bounded subanalytic manifold (i. e. we do not assume that the closure is a pseudomanifold) while Lefschetz duality for intersection homology  is usually stated on pseudomanifolds with boundary.  This lead the author to the conclusion that there must be a Lefschetz duality in a slightly more general setting than the framework of pseudomanifolds with boundary in the way that they are usually defined.  The Lefschetz duality  for $\partial$-pseudomanifolds that we develop in this paper is indeed the exact geometric counterpart of the duality observed in \cite{vl1} (on $\pa$-pseudomanifolds).

Part of the problem of Lefschetz duality for intersection homology is  that the intersection homology of the  pair $(X;\partial X)$ is not well defined since the allowable chains of the boundary are not necessarily allowable in  $X$. Allowability is a condition which depends on the choice of a perversity (see definitions below).  We explain how, given a perversity $p$, we can construct a perversity $\p$ for the boundary in a natrual way which makes it possible to extend  Lefschetz duality.

We shall work in the subanalytic framework. In \cite{gm1}, the authors  prefer to work in the PL category but, as they themselves emphasize in the introduction, evertyhing could have been carried out in the subanalytic category. We  avoided sheaf theory, striving to make the proof as elementary as possible. Although  the arguments  presented in \cite{gm1,gm2} (for proving  Poincar\'e duality) seem to apply for proving our theorem,  we shall present a different argument.  Our proof is nevertheless based of their construction of the natural paring.

\subsection*{Content of the paper} In the first section we recall the definitions of intersection homology and stratified pseudomanifold.  In the second section we introduce our notion of stratified $\pa$-pseudomanifold and extend the basic notions to this setting, introducing our "boundary perversity".   We then extend the intersection pairing of \cite{gm1} to our setting. We finally   establish Lefschetz duality for $\pa$-pseudomanifolds,  starting with  some local computations of the intersection homology groups and then  gluing the local information in a fairly classical way.

\subsection*{Some notations and  conventions}  By
"subanalytic" we mean "globally subanalytic", i. e. which remains
subanalytic after  compactifying  $\R^n$ (by $
\mathbb{P}^n$).   Balls in $\R^n$ are denoted $B(x_0;\ep)$ and are considered for the norm $\underset{i\leq n}{\sup} |x_i|$.

Given a set
$X\subset \R^n$, we denote by $C^j(X)$ the singular cohomology
cochain complex. Simplices are defined as
continuous subanalytic maps $\sigma :\Delta _j \to X$, where
$\Delta_j$ is the standard simplex. The coefficient ring will be always be $\R$. We denote by $X_{reg}$ the regular locus of $X$, i. e. the set of points at which $X$ is a manifold (without boundary) of dimension $\dim X$. We denote by $X_{sing}$ its complement in $X$.

%A {\bf stratification} of a set $X\subsets \R^n$ is a locally finite partition of $X$ into submanifolds of $\R^n$ (called strata).

\section{Intersection homology }
 We recall the definition of intersection homology as it was introduced by M. Goresky and R. Macpherson \cite{gm1,gm2}.

\begin{defs}\label{del pseudomanifolds}
%A $n$-dimensional pseudomanifold is a semiagebraic $X$
A subanalytic  subset $X\subset \R^n$ is an  {\bf $l$-dimensional pseudomanifold} if
 $X_{reg}$ is an $l$-dimensional manifold which is  dense in $X$ and  $\dim X_{sing}<l-1$.

A {\bf stratified pseudomanifold} is the data of an $l$-dimensional pseudomanifold $X$ together with  a filtration: $$\emptyset=X_{-1}\subset X_0\subset \dots \subset  X_{l}=X,$$ with $X_{l-1}=X_{l-2}$, such that the subsets  $X_i\setminus X_{i-1}$  constitute a locally topologically trivial stratification.
\end{defs}

\begin{dfn}\label{dfn_pseudo_a_bord}
A {\bf stratified  pseudomanifold with boundary} is a subanalytic couple $(X;\pa X)$ together with a subanalytic  filtration $$\emptyset=X_{-1}\subset X_0\subset \dots \subset X_{l-1} \subset X_{l}=X,$$ %and $$
%\emptyset=X_{-1}\subset X_0\subset \dots \subset X_{m-2} = X_{m-1}=X\setminus A,$$%\end{equation}
and
%\begin{equation}\label{eq_filtr_A}
% A_0\subset \dots \subset A_{m-2} = A_{m-1}=A,
%\end{equation}
 \begin{enumerate}
\item  $X\setminus \pa X$ is a  $l$-dimensional stratified pseudomanifold (with the filtration $X_j\setminus \pa X$).
\item  $\pa X$ is a  stratified pseudomanifold (with the filtration $X_j':=X_j\cap \pa X$)
\item $\pa X$ has a {\bf stratified collared neighborhood}: there exist a neighborhood $U$ of $\pa X$ in $X$ and a  homeomorphism $h:  \pa X \times [0;1]\to U $ such that $h(U \cap X_j)=X_{j-1}'\times [0;1]$.
\end{enumerate}
\end{dfn}

\begin{dfn}
A {\bf perversity} is a sequence of integers $p = (p_2, p_3,\dots, p_l)$
such that $p_2 = 0$ and $p_{k+1} = p_k$ or $p_k + 1$. A subanalytic subspace
$Y \subset X$ is called  {\bf $(i;p)$-allowable} if $\dim Y \cap
X_{l-k} \leq p_k+i-k$.  Define $I^{p}C_i (X)$ as the subgroup of
$C_i(X)$ consisting of the subanalytic chains $\sigma$ such that $|\sigma|$
is $(p, i)$-allowable and $|\partial \sigma|$ is $(p, i -
1)$-allowable.

 The {\bf $i^{th}$ intersection homology group of perversity $p$}, denoted
$I^{p}H_j (X)$, is the $i^{th}$ homology group of the  chain
complex $I^{p}C_\bullet(X).$  %The {\bf $i^{th}$ intersection cohomology group of perversity $p$}, denoted
%$I^{p}H^j (X)$, is defined as $Hom (I^{p}H_j (X);\R)$.

The {\bf Borel-Moore intersection chain complex} $I^p C_j ^{BM}(X)$ is defined as the chain complex constituted by the locally finite $p$-allowable subanalytic chains. We denote by $I^p H_j ^{BM} (X)$ the Borel-Moore intersection homology groups.
\end{dfn}

\medskip
\subsection{Lefschetz-Poincar\'e duality for pseudomanifolds with boundary.}\label{sect_ih}

 We denote by $t$ {\bf
the maximal perversity}, i. e. $t=(0;1;\dots;l-2)$. Two perversities $p$ and $q$ are said {\bf complement} if $p+q=t$.

\medskip

\begin{thm}(Generalized Lefschetz-Poincar\'e duality
\cite{gm1,gm2,f2,f3,k})\label{thm_poincare_ih}
Let $X$ be a subanalytic compact oriented stratified pseudomanifold with boundary $\pa X$. For any complement
perversities   $p$ and $q$:
$$I^p H_j(X\setminus \pa X)=I^q H_{l-j}^{BM}(X\setminus \pa X).$$
\end{thm}

%If $X$ is compact and $U$ is the collared neighborhood  of the origin provided by the definition of stratified pseudomanifold with boundary, then $I^pH _j ( U)I^pC _j (X) $ Lefschetz  duality takes the following form:
%$$I^pH _j (X\seminus U) \simeq I^q H_j (X;U).$$

\section{$\partial$-Pseudomanifolds.}
 We first introduce the notion of $\partial$-pseudomanifold and then naturally extend  intersection homology to these spaces. Basically, we drop the assumption $(3)$ of having a collared neighborhood (see Definition \ref{dfn_pseudo_a_bord}). Let $X$ be a subanalytic set of dimension $l$.

\begin{dfn}
The {\bf $\partial$-regular locus} of $X$ is the set of points near which the set $X$ is a  manifold with nonempty boundary. We will denote it by $X_{\partial,reg}$.
The closure of $X_{\partial,reg}$  will be called {\bf the boundary of $X$} and will be denoted $\partial X$.
%Observe that $\partial X_{reg}$ coincides with $(\partial X)_{reg}$.
The set $X$ is said to be a  {\bf $\partial$-pseudomanifold}   if $ X\setminus \pa X$ is a pseudomanifold and if  $\dim \partial X\setminus X_{\pa,reg}<l-2$.
\end{dfn}

\begin{exa}
It follows from the definition that if  $X$ is a pseudomanifold then it is a $\partial$-pseudomanifold (with empty boundary). %This will be of service.

 Let us give another example. Let  $f:\R^n \to \R$ be a subanalytic map such that $\dim Sing(f)\cap \{ f=0 \}<n-2$. Then $\{x \in \R^n: f(x) \geq 0\}$ is a subanalytic $\partial$-pseudomanifold.
\end{exa}

Of course if $X$ is a $\partial$-pseudomanifold and  $\partial X$ has a collared neighborhood, then it is a pseudomanifold with boundary in the usual sense. %(since we are in the subanalytic framework we can find two stratifications).
Nevetherless, the above example shows that a $\partial$-pseudomanifold does not always admit a collared neighborhood. It follows from the definitions that $\pa X \subset X_{sing}$.

  We will show that Lefschetz duality holds for $\pa$-pseudomanifolds.
We give an example (a double pinched torus in $S^3$) in the last section.

\subsection{Stratified $\partial$-pseudomanifolds.}
\begin{dfn}\label{dfn_pa_pseudo}
A subanalytic $\partial$-pseudomanifold $X$ is {\bf stratified} if there exists a subanalytic  filtration:
$$\emptyset =X_{-1}\subset X_0\subset \dots \subset X_{l-1} \subset X_l=X,$$
with $X_{i}\setminus X_{i-1}$ toplogically trival stratification compatible with $\pa X$ and such that:
 \begin{enumerate}
\item  $X\setminus \partial X$ is a stratified pseudomanifold (with the filtration $X_j\setminus \partial X$).
\item  $\partial X$ is a  stratified pseudomanifold (with the filtration $X_j':=X_j\cap \partial X$).
\end{enumerate}

 \end{dfn}
%\begin{dfn}
%A $\partial$-pseudomanifold is {\bf stratified} if there exists a filtration:
%$$X_0\subset \dots \subset X_{l-1} \subset X_l,$$
%such that $\dim X_{l-1} \setminus  \partial X \leq l-2$, $\dim X_{l-1}\setminus \partial X=l-2$  and such that  $\partial X$ is a union of strata.    In this case,  the filtration of $X$ gives rise to a filtration of $ \partial X$ that we will denote by:
%$$X_0 '\subset \dots \subset X_{l-1}'=\partial X.$$
% \end{dfn}

If we compare with the definition of pseudomanifolds with boundary, we see that the assumption $(3)$ about the stratified collared neighborhood has  been dropped.

 Subanalytic  $\pa$-pseudomanifolds can always be stratified.
We now define the intersection homology of a $\pa$-pseudomanifold. It extends naturally Goresky and MacPherson's definition.
% It is useful for our purpose only if we have an exact sequence for the pair.

\subsection{Intersection homology of a $\pa$-pseudomanifold.}
\subsection*{The boundary perversity $\p$.}
Given an $l$-perversity $p$, define an $(l-1)$-perversity by:
$$\check{p}_j:=p_{j+1}-p_3, $$
for  $j\geq 2$. It is easily checked from the definition that  $\p$ is a $(l-1)$-perversity. %We call this perversity, {\bf the boundary perversity}.

 Note that  $p$ and $q$ are complement $l$-perversities iff $\p$ and $\check{q}$ are complement $(l-1)$-perversities.

\begin{exa}
Denote respectively by $0^l$ and $t^l$ the zero and top  $l$-perversities. We have $\check{0}^l=0^{l-1}$, $\check{t}^l=t^{l-1}$. The middle perversities are interchanged by $\check{ }$ in the sense that  $\check{n}^l=m^{l-1}$, $\check{m}^l=n^{l-1}$.
\end{exa}

\subsection*{The intersection homology groups.} Denote by $\Sigma$ a subanalytic  stratification of a subanalytic  $\pa$-pseudomanifold $X$ and by $\Sigma '$ the induced  stratification of $\partial X$ (see Definition \ref{dfn_pa_pseudo} $(2)$).  Fix a perversity $p$.

We say that $Y\subset X$ {\bf is $(j;p)$-allowable} (with respect to $(\Sigma;\Sigma ')$) if $$\dim cl(Y\setminus  \partial X)\cap X_{l-k}\leq j-k+p_k,$$  and if $ Y \cap  \partial X $ is $(j ;\p)$-allowable (w. r. t. $\Sigma '$). A {\bf $j$-chain $\sigma$  is $p$-allowable} if $|\sigma|$ is $(j;p)$-allowable.

Let $I^p C_j(X)$ be the chain subcomplex of $C_j(X)$ constituted by the $p$-allowable $j$-chains $\sigma$ for which  $\pa  \sigma$ is $p$-allowable.
If $X$ is a pseudomanifold then of course this chain complex coincides with the one introduced in \cite{gm1}.

\subsection*{Relative intersection homology of $\pa$-pseudomanifolds.}
The relative intersection homology groups are of importance for Lefchetz duality.

Observe that it follows from this definition that $I^{\p} C_j(\partial X) \subset I^p C_j( X)$ and hence we may set:
 $$I^p C_j (X;\pa X):=\frac{I^p C_j(X)}{I^{\p} C_j(\partial X)}.$$

%We denote by $I ^p H_j (X)$ the  cohomology groups of the chain complex $I ^p C_j(X)$.

  As usual we have the following long exact sequence:
$$\dots  \to I^{\p}  H_j( \partial X) \to  I^{p} H_j (X)\to I^p H_j (X; \partial X) \to I^{\p} H_{j-1} ( \partial X) \to \dots .$$

\subsection*{Borel-Moore intersection homology groups for $\pa$-pseudomanifolds.}
 The Borel-Moore chain complex, denoted $I^p C_j^{BM}(X)$, are defined as the locally finite combinations of allowable simplices (with subanalytic support). We denote by $I^pH_j^{BM}(X)$ the resulting homology groups. For any subanalytic open subset $W$ of $X$, define also $I^p C_j^{BM}(X;W)$ as the chain complex constituted by the chains $\sigma \in I^pC_j^{BM}(X)$ such that $|\sigma| \cap W=\emptyset $. Denote by $I^p H_j(X;W)$ the coresponding homology groups.

%%%%%%%%%%%%%%%%%%%%%%%%%%%%%%%%%%%%%%%%%%%%%%%%%%%%%%%%%%%%%%%%%%%%%%%%%%%%%%%%%%%%%%%%%%%%%%%%%%%%%

\section{Two preliminary Lemmas}
%\subsection{Three preliminary lemmas}
 Let $X$ be an oriented locally closed conected subanalytic   stratified  $\partial$-pseudomanifold.

\subsection{A local exact sequence.} We shall need the following local exact sequence for the local computation of the homology groups. The material of this section is quite classical.

\begin{lem}\label{lem_local_exact_sequence}
Let  $x_0 \in X$ and set $X^\ep:=B(x_0;\ep) \cap X$. For $\ep>0$ small enough there is a long exact sequence:
\begin{equation}\label{eq_loc_exact_sequence}\dots \to I^p H_j (X^\ep) \to I^p H_j ^{BM}(X^\ep) \to I^p H_{j-1} (X^\ep \setminus x_0) \to \dots . \end{equation}
\end{lem}
\begin{proof}
%Set $Z^\ep:=\{x \in X^\ep: |x-x_0|>\frac{\ep}{2}\}$ and
Observe that we have an exact sequence:
$$ \dots \to I^p H_j (X^\ep) \to I^p H_j (X^\ep;X^\ep\setminus x_0) \to I^p H_{j-1} (X^\ep\setminus x_0 ) \to \dots . $$
Due to  the local conic structure of subanalytic sets,  $X^\ep \setminus x_0$ retracts by deformation onto $S(x_0;\ep)\cap X$ and    we have an isomorphism:
$$ I^p H_j (X^\ep;X^\ep\setminus x_0)\simeq  I^p H_j ^{BM}(X^\ep) , $$
which yields the result.
\end{proof}

%\begin{rem}The same holds true for the relative homology (with the same proof since the local conic structure preserves the boundary), at any $x_0 \in \pa X$. \end{rem}
%\begin{equation}\label{eq_loc_exact_sequence}\dots \to I^p H_j (X^\ep;\partial X^\ep) \to I^p H_j ^{BM}(X^\ep;\partial X^\ep) \to I^p H_{j-1} (X^\ep \setminus x_0; \partial X^\ep \setminus x_0) \to \dots . \end{equation}

%
%In the above exact sequences the maps  $I^p H_j ^{BM}(X^\ep) \to I^p H_{j-1} (X^\ep \setminus x_0)$ are induced by  the boundary operator. This follows from the construction. The same is true for the relative version (\ref{eq_loc_exact_sequence}).

\begin{lem}\label{lem_produits} Let $p'$ be the $(l-1)$ perversity defined by $p'_i:=p_i$ if $i \leq l-1$. Then
\begin{equation}I^{p'} H_j (X)=I^p H_j (X\times (0;1)),\end{equation}
(with the product stratification) and the same holds true for $(X;\partial X).$
\end{lem}
\begin{proof}
The inclusion $i: X  \to X \times (0;1)$, $x \mapsto (x;\frac{1}{2})$ clearly sends a $p'$ allowable chain onto a $p$ allowable chain. It induces an isomorphism between the respective homology groups, as well as between the relative intersection homology groups.
\end{proof}

\begin{rem}\label{rem_produit}
For Borel Moore homology an analogous statement holds true:
$$I^{p'} H ^{BM}_{j} (X)=I^p H_{j+1} ^{BM} (X\times (0;1)).$$
\end{rem}

%\section{Lefschetz duality}
%\subsection{Pairings on $\partial$-manifolds.}Fix two perversities $p$  and $q$ with $p+q=t$.

%\begin{lem}\label{lem_paring}
%There is a well defined intersection pairing:
%$$I^p H_i (X)\times I^q H_{l-i}^{BM} (X;\partial X)\to I^t H_{0}(X).$$
%such that $[C\cap D]=[C]\cap [D]$.
%\end{lem}

\section{Local computations of $IH$.} As in the case of pseudomanifolds \cite{gm2,k}, the most important step is to compute the homology groups locally. This already yields Lefschetz duality "locally". In  section \ref{sect_lef} we shall glue this local information to establish Lefschetz duality globally.
The local computation is quite classical.
 Let $X$ be a subanalytic stratified $l$-dimensional $\partial$-pseudomanifold.
\begin{lem}\label{lem_local}
For any perversity $p$, the mappings  $I^p H_j (X^\ep \setminus x_0)\to I^p H_j(X^\ep),$ and $I^p H_j (X^\ep \setminus x_0;\pa X^\ep \setminus x_0)\to I^p H_j(X^\ep;\pa X^\ep ),$ induced by inclusion,  are onto. The boundary operator  $I^p H_j^{BM} (X^\ep )\to I^p H_j^{BM}(X^\ep\setminus x_0)$ constructed in Lemma \ref{lem_local_exact_sequence} is one-to-one.
\end{lem}
\begin{proof}
Let  $\sigma \in I^p H_j (X^\ep)$ be a nonzero class. Then $|\sigma|$ does  not contains $x_0$ (since otherwise $\sigma$ could be retracted onto $x_0$). Hence it lies in $X^\ep \setminus x_0$. This argument also applies to the relative homology  and the assertion on Borel-Moore homology is a consequence of Lemma \ref{lem_local_exact_sequence}. \end{proof}

\begin{lem}\label{lem_calcul local} Let $x_0  \in \partial X \cap X_0$ and set  $X^\ep :=X\cap B(x_0;\ep)$.  For any $\ep>0$ small enough:
\begin{enumerate}
\item If $p_3=0$  then:$$I^p H_j(X^\ep)\simeq \left\{
\begin{array}{ll}I^p H_j (X^\ep \setminus x_0),\qquad  \mbox{if}
\quad p_{l} < l-j-2, \\   0,
 \qquad  \hskip 2.2cm \mbox{if} \quad p_l>l-j-2.
\end{array}
\right.$$

\bigskip

\item  If $p_3=1$ then:$$I^p H_j(X^\ep)\simeq \left\{
\begin{array}{ll}I^p H_j (X^\ep \setminus x_0),\qquad  \mbox{if}
\quad p_{l} < l-j-1, \\   0,
 \qquad  \hskip 2.2cm \mbox{otherwise} .
\end{array}
\right.$$

\bigskip

\item If $p_3=0$  then:$$I^p H_j(X^\ep;\pa X^\ep)\simeq \left\{
\begin{array}{ll} I^p H_j (X^\ep \setminus x_0;\pa X^\ep \setminus x_0),\qquad  \mbox{if}
\quad p_{l} \leq l-j-2, \\   0,
 \qquad  \hskip 3.8cm \mbox{otherwise.}
\end{array}
\right.$$

\bigskip

%%%%%%%%%%%%%%%%%%%%%%%%%%%%%%%%5

\item If $p_3=1$  then:$$I^p H_j(X^\ep;\pa X^\ep)\simeq \left\{
\begin{array}{ll} I^p H_j (X^\ep \setminus x_0;\pa X^\ep \setminus x_0),\qquad  \mbox{if}
\quad p_{l} < l-j-1, \\   0,
 \qquad  \hskip 3.8cm \mbox{if} \quad  p_{l} > l-j-1.
\end{array}
\right.$$

 %I^p H_j(X^\ep \setminus x_0;\pa X^\ep \setminus x_0),$$ if
%$p_{l} < l-j-1$, and $$ I^p H_j(X^\ep;\pa X^\ep)\simeq 0,$$
%if $$.
\end{enumerate}

Furthermore, the isomorphisms  are induced by the  natural inclusions.
\end{lem}
\begin{proof}
%Thanks to Lemma \ref{lem_produits}, the Lemma reduces  to the case $k=0$ (recall that strata are locally trivial).
%
In every case we only check injectivity since surjectivity is a consequence of Lemma \ref{lem_local}.

{\it Proof of $(1)$.} Assume  $p_3=0$.

Suppose  that $p_{l} < l-j-2$, consider a cycle  $\sigma$  of $I^p C_j (X^\ep \setminus x_0)$, and assume that  $\sigma=\partial \tau$ for some  $\tau \in I^p C_{j+1} (X^\ep )$. Then,  as $|\tau |\cap \partial X$ is $(\p ;j+1)$-allowable,
$$\dim |\tau|\cap X_{0} \cap \partial X \leq j+1-(l-1)+p_{l}<0.$$
%Moreover the allowability condition also demands:
%$$X_0\cap cl (|\tau|\setminus \pa X)\leq j+1-l +p_l<-1. $$
This entails that $|\tau|$ may not contain $x_0$, showing that the map induced by inclusion $I^p H_j (X^\ep \setminus x_0)\to I^p H_j(X^\ep)$ is one-to-one.

%Take now a cycle $\tau$  of $X^\ep \setminus x_0$ and assume that $\tau=\partial \tau'$ for some $(p;j+1)$-allowable chain $\tau'$. Then, for the same reasons,  the support of $\tau'$ cannot meet $x_0$, showing that the class of $\sigma$ is zero in $ I^p H_j (X^\ep \setminus x_0)$ as well.

%If $p_{l} = l-j-2$

If $p_{l} > l-j-2$ then $(j+1)-l+p_l \geq  0$. This means that the support of a $(j+1;p)$  allowable chain may contain the point $x_0$. Thus the retraction by deformation  to $x_0$ of  any     $\sigma \in I^p H_j(X^\ep)$ gives rise to a chain $\tau \in I^p C_{j+1} (X)$ such that $\sigma=\pa \tau $ in $I^p H_j(X^\ep)$. This  shows that  $I^p H_j(X^\ep)$ is zero in this case.

{\it Proof of $(2)$.}  Suppose now $p_3=1$.

 Assume first that $p_{l} < l-j-1$, consider a cycle  $\sigma$  of $I^p C_j (X^\ep\setminus x_0)$, and suppose that there is a  $\tau$  in $I^p C_{j+1} (X^\ep)$ with $\sigma=\partial \tau$.

 Then,  as $|\tau |\cap \pa X$ is $(j+1;\p)$ allowable,
$$\dim |\tau|\cap X_{0}\cap \pa X \leq j+1-(l-1)+\p_{l-1}<0.$$
%Moreover the allowability condition also demands:
%$$X_0\cap cl (|\sigma|\setminus \pa X)\leq j+1-l +p_l<0. $$
%Hence, $\sigma$ may not intersect $X_0$ and thus is a nonzero cycle of  $X^\ep \setminus x_0$.% showing that the map induced by inclusion $I^p H_j (X^\ep \setminus x_0)\to I^p H_j((X^\ep)$ is onto.
The same applies to $cl(|\tau|\setminus \pa X^\ep)$.  This entails that $|\tau|$ may not contain $x_0$, showing that the map induced by inclusion $I^p H_j (X^\ep \setminus x_0)\to I^p H_j(X^\ep)$ is one-to-one.

If $p_{l} \geq l-j-1$ then $(j+1)-(l-1)+\p _{l-1}  \geq 0$ and  $(j+1)-l+p _l  \geq 0$. This means that $(j+1;p)$  allowable chains may meet $x_0$. The retraction  to $x_0$ of  any     $\sigma \in I^p H_j(X^\ep)$ gives rise to a chain $\tau \in I^p C_{j+1} (X^\ep)$ such that $\sigma=\pa \tau $ in $I^p H_j(X^\ep)$. This  shows that  $I^p H_j(X^\ep)$ is zero.

%  Suppose now that $p_{l} < l-j-1$ and consider a cycle  $\sigma$  of $I^p C_j (X^\ep)$. Then,  as $|\sigma |\cap \partial X$ is $(\p ;j)$-allowable,
%$$\dim |\sigma|\cap X_{0} \cap \partial X \leq j-l+1+p_{l}<-1.$$
%Moreover the allowability condition also demands:
%$$X_0\cap cl (|\sigma|\setminus \pa X)\leq j-l +p_l<-2. $$
%Hence, $|\sigma|$ may not contain $x_0$, showing that the map induced by inclusion $I^p H_j (X^\ep \setminus x_0)\to I^p H_j((X^\ep)$ is onto.

\bigskip

 We now consider the relative homology.

\medskip

{\it Proof of $(3)$.} Assume $p_3=0$.

 If $p_{l} > l-j-2$ then $j-(l-1)+\p_{l-1}\geq 0$ and $(j+1)-l+p_l \geq  0$. This means that the support of any element of $ I^pC_{j+1} (X^\ep;\pa X^\ep)$ may contain the point $x_0$. Consequently, the retract by deformation $\tau$ of any $\sigma \in I^p H_j (X^\ep;\pa X^\ep)$ is $p$-allowable,  showing that  $ I^p H_j(X^\ep;\partial X^\ep)$ is zero in this case.

% Moreover  as $j-(l-1)+p_l$ is nonegative as well and $j$ chains of $I^p C_j(\pa X^\ep)$ are $\p$-allowable and thus are $p$-allowable in $X$.   Thus, if $\sigma \in I^p H_j (X^\ep;\pa X^\ep)$ then $\sigma$ may be retracted to $x_0$ and the retract $\tau$ as well as  its boundary $\partial \tau$  are $p$-allowable. Hence,  $ I^p H_j(X^\ep;\partial X^\ep)$ is zero in this case.

If now $p_l\leq l-j-2$, let $\sigma \in I^p H_j(X^\ep\setminus x_0;\pa X^\ep\setminus x_0)$ and let $\tau \in I^p C_{j+1}(X^\ep;\pa X^\ep)$ be such that $\pa \tau =\sigma$. As $\tau$ is $p$-allowable we have $$\dim cl(|\tau|\setminus \pa X^\ep)\cap X_0\leq j+1-l+p_l\leq -1.$$
Therefore, there is a small neighborhood of $x_0$ in $X^\ep$ such that:
$$U\cap |\tau| \subset \pa X^\ep.$$

%then $j-l+p_l \leq -2$ and $j-(l-1)+\p_{l-1}\leq-1$ and thus no   p-allowable $j$-cycle of $I^p C_j(X^\ep;\pa X^\ep)$  may contain $x_0$,  showing that the map induced by inclusion $I^p H_j (X^\ep \setminus x_0;\pa X^\ep \setminus x_0)\to I^p H_j((X^\ep;\pa X^\ep)$ is onto. Let us show that it is one-to-one.

 Subdividing the simplices, we may assume that all those (of the chain $\tau$) which contain the point $x_0$ fit in $U$. As they are all zero in $ I^p C_{j+1}(X^\ep;\pa X^\ep)$ we can drop them without affecting the fact that $\pa \tau =\sigma$ in  $I^p H_{j}(X^\ep;\pa X^\ep)$. In other words, we can assume that $\tau \in  I^p C_{j+1}(X^\ep\setminus x_0;\pa X^\ep\setminus x_0)$, as required.

{\it Proof of $(4)$}. Assume that  $p_3=1$.

 Take $\sigma \in I^p H_j(X^\ep\setminus x_0;\pa X^\ep\setminus x_0)$ and let $\tau \in I^p C_{j+1}(X^\ep;\pa X^\ep)$ be such that $\pa \tau =\sigma$. If $p_l<l-j-1$ then  $(j+1)-l+p_l <0$. This entails that $|\tau|$ may not contain $x_0$, showing that the map induced by inclusion $I^p H_j (X^\ep \setminus x_0;\pa X^\ep \setminus x_0)\to I^p H_j(X^\ep;\pa X^\ep )$ is one-to-one.

\medskip

If now $p_l>l-j-1$, then $j+1-l+p_l >0$ and $j-(l-1)+\p_{l-1}>-1$ and thus the retraction by deformation to $x_0$ of  any     $\sigma \in I^p H_j(X^\ep;\pa X^\ep)$ gives rise to a chain $\tau \in I^p C_{j+1} (X^\ep;\pa X^{\ep})$ such that $\sigma=\pa \tau $ in $I^p H_j(X^\ep;\pa X^\ep)$. This  shows that  $I^p H_j(X^\ep;\pa X^\ep)$ is zero.
\end{proof}

\begin{rem}\label{rem_calcul}
Thanks to the exact sequence (\ref{eq_loc_exact_sequence})   we may derive from the above lemma that
\begin{enumerate}
\item If $p_3=0$  then:$$I^p H_j^{BM}(X^\ep)\simeq \left\{
\begin{array}{ll}  I^p  H_{j-1} (X^\ep \setminus x_0),\qquad  &\mbox{if}
\quad p_{l} > l-j-1, \\   0,
 &\mbox{if} \quad  p_{l} < l-j-1.
\end{array}
\right.$$

%$$I^p H_{l-p_l-1}^{BM}(X^\ep)\simeq Im [I^p  H_{l-p_l-2} (\pa X^\ep \setminus x_0)\to I^p  H_{l-p_l-2} (X^\ep \setminus x_0)],$$
%where $Im$ stands for image and the arrow is induced by inclusion.
\item If $p_3=1$  then:$$I^p H_j^{BM}(X^\ep)\simeq\left\{
\begin{array}{ll}  I^p H_{j-1} (X^\ep \setminus x_0)\qquad  &\mbox{if}
\quad p_{l} > l-j-1, \\   0,
 &\mbox{otherwise.}
\end{array}
\right.$$
\end{enumerate}
Furthermore, the isomorphism is induced by the boundary operator of the exact sequence (\ref{eq_loc_exact_sequence}).
\end{rem}

The case where $j=l-p_l-1$ and $p_3=1$ is more delicate and is adressed separately in the following lemma.

\begin{lem}\label{lem_calcul_image}
Let $p$ and $q$ be complement perversities with $p_3=1$ and set $j=l-p_l-1$. Let $x_0 \in \pa X\cap X_0$ .
\begin{enumerate}
\item Let  $$b:I^p H_{j} (X^\ep \setminus x_0;\pa X^\ep \setminus x_0)\to I^p H_{j} ( X^\ep ;\pa X^\ep)$$ be the map induced by inclusion. Then $$\ker b= \ker \, \pa $$ where $$\pa :I^p H_{j} (X^\ep \setminus x_0;\pa X^\ep \setminus x_0)\to I^p H_{j-1} ( \pa X^\ep \setminus x_0) $$ is induced by the boundary operator.
\item Let  $$b':I^q H_{l-j} ^{BM} (X^\ep)\to I^q H_{l-j}^{BM} ( X^\ep  \setminus x_0)$$ be the natural map. Then $$Im\, b'= Im\, i_* ,$$ where $$i_* :I^q H_{l-j} ^{BM} (\pa X^\ep \setminus x_0)\to
        I^q H_{l-j}^{BM} ( X^\ep \setminus x_0)$$ is induced inclusion.
\end{enumerate}

\end{lem}
\begin{proof}
{\it Proof of $(1)$.} Consider $\sigma \in \ker \pa $. Oberve that $(j+1)-l+p_l= (j+1)-(l-1)+\p_{l-1} =0$. Therefore, a $p$ allowable $(j+1)$-chain may contain the point $x_0$ (but not at a boundary point). Let $\tau$ be the chain obtained by retracting by deformation $\sigma$ onto $x_0$.   Then,   $\pa \tau=\sigma$ (since $\sigma \in \ker \pa$) meaning that $\sigma \in \ker b$.  Thus, $\ker \pa \subset \ker b$. Let us show the reversed inclusion.

Take now  $\sigma \in \ker b$ and let $\tau \in I^p C_{j+1}(X^\ep;\pa X^\ep)$ be such that $\pa \tau =\sigma$ in $I^p H_j(X^\ep;\pa X^\ep)$. As $|\pa \tau|$ is $(j;p)$-allowable we have $$\dim |\pa \tau|\cap X_0\leq j-(l-1)+\p _{l-1}= -1.$$
Therefore, $|\pa \tau|$ cannot contain $x_0$.
%$$U\cap |\pa \tau| \setminus \pa X=\emptyset.$$
Let $c \in I^p C_j (\pa X\setminus x_0)$ be the chain constituted by the simplices of $\pa \tau$ that lie in
$\pa X$. For a suitable representative $d$ of the class $\sigma$ we have $\pa \tau= c+d$ (as $\sigma$ is a relative chain, we may drop all the simplices of $\sigma$ that lie in $\pa X^\ep$ without changing the class).   This entails that $\pa c=-\pa d$ and, since $c \in I^p C_j (\pa X^\ep\setminus x_0)$,  that $\pa \sigma $ is zero in  $I^p C_{j-1} (\pa X\setminus x_0)$, as required.

{\it Proof of $(2)$.} If $p_l=l-j-1$ then $q_l=j-1$. We claim that the map $I^q H_{l-j} ^{BM} (\pa X^\ep )\overset{\theta}{\to}  I^q H_{l-j}^{BM} ( X^\ep ) $, induced by inclusion, is onto. Indeed, if $\sigma \in  I^q H_{l-j}^{BM} ( X^\ep ) $ then $$\dim cl(|\sigma|\setminus \pa X)\cap X_0\leq (l-j)-l+q_l= -1.$$
Therefore, there is a small neighborhood of $x_0$ in $X^\ep$ such that:
$$U\cap |\tau| \subset \pa X^\ep.$$
The retraction by deformation of  the complement of this neighborhood onto the link provides a $q$-allowable Borel-Moore chain.  Substracting the boundary of this chain provides  a representative of the class $\sigma$ which lies in $\pa X^\ep$. This shows that $\theta$ is onto, as claimed.

 Observe that, since $\pa X^\ep$ is a pseudomanifold, the map $I^q H_{l-j} ^{BM} (\pa X^\ep)\to
        I^q H_{l-j}^{BM} ( \pa X^\ep \setminus x_0)$, induced by inclusion, is  an isomorphism (see \cite{gm2}, as  $q_l=j-1$ we have $l-j=l-1-\check{q}_{l-1}$).   Now, the lemma follows from the commutative diagram below:
\begin{center}
     \begin{picture}(-210,0)
\put(-120,3){$\theta$}
\put(-120,-47){$i_*$}
\put(-180,-25) {$\simeq$}
\put(-65,-25){$b'$}
      \put(-230,-3){$I^q H_{l-j}^{BM} (\pa X^\ep )$}
        \put(-90,-3){$I^q H_{l-j} ^{BM}( X^\ep)$}
      \put(-230,-52){$I^q H_{l-j} ^{BM} (\pa X^\ep \setminus x_0)$}
          \put(-90,-52){$I^q H_{l-j}^{BM} ( X^\ep \setminus x_0 )$}
%      \put(-90,15){$ A_{1} $}
%      \put(-90,-45){$A_2$}
      \put(-140,-50){\vector(3,0){40}}
      \put(-160,0){\vector(3,0){60}}
  \put(-185,-10){\vector(0,-1){30}}
      \put(-70,-10){\vector(0,-1){30}}
     % \put(-130,-40){$h_2$}
%\put(-130,15){$h_1$}
     \end{picture}
    \end{center}
\vskip 2cm
\end{proof}
%\begin{pro}\label{pro_lefschetz_local}
%\end{pro}

\section{Intersection pairings on pseudomanifolds.} In \cite{gm1} the authors define an intersection pairing on stratified pseudomanifolds, which is dual to the cup product up to some isomorphisms induced by excision. Let us recall their construction and then see how it fits with our setting.

\subsection{Pairings on pseudomanifolds} Let $X$ be an oriented $l$-dimensional subanalytic  stratified pseudomanifold (without boundary).

\begin{dfn}\label{dfn_dim_trans}
 Let $p$, $q$ and $r$  be three perversities with $p+q=r$. We say that $C \in I^p C_i (X)$ and  $D \in I^q C_j (X)$ are {\bf dimensionally transverse} if $|C| \cap |D|$ is $(i+j-l;r)$ allowable. Denote it by $C \pitchfork D$.
\end{dfn}

Given two dimensionally transverse chains $C$ and $D$, the authors define in \cite{gm1} an intersection pairing as follows. Let $J:=|\pa C|\cup |\pa D|\cup X_{sing}$. Let $\overline{C} \in  H_i(|C|, |\pa C|)$ and $ \overline{D}\in  H_j(|D|, |\pa D|)$ be the classes
determined by $C$ and $D$. Define $C \cap D$ to be the chain determined by
the image of $(\overline{C}, \overline{D})$ under the following sequence of homomorphisms:

\begin{center}
     \begin{picture}(90,210)
      \put(-10,200){$ H_i(|C|; |\pa C|) \times H_j(|D|; |\pa D|) $}
%      \put(-90,-45){$A_2$}
      \put(50,195){\vector(0,-1){15}}
   \put(-10,170){$ H_i(|C|; | C|\cap J) \times H_j(|D|; | D|\cap J) $}
 \put(50,165){\vector(0,-1){25}}
 \put(30,150){$\simeq$}
 \put(60,150){(excision)}

 \put(-10,125){$ H_i(|C|\cup J;  J) \times H_j(|D|\cup J;  J) $}

 \put(50,120){\vector(0,-1){25}}
 \put(60,105){$\cap [X]\times \cap [X]$}
 \put(130,105){(duality)}
 \put(30,105){$\simeq$}

 \put(-60,85){$ H^{l-i}(X \setminus  J; X \setminus   (|C|\cup J)) \times H^{l-j}( X\setminus  J;X \setminus (|D|\cup J)) $}

 \put(50,80){\vector(0,-1){25}}
 \put(60,65){(cup product)}

 \put(-10,45){$ H^{2l-i-j}(X \setminus  J; X \setminus   ((|C|\cap |D|)\cup J)) $}

 \put(50,40){\vector(0,-1){25}}
 \put(60,25){$\cap [X]$}
 \put(130,25){(duality)}
 \put(30,25){$\simeq$}

  \put(-10,0){$ H_{i+j-l}((|C|\cap |D|)\cup J;  J) $}
 \put(50,-30){\vector(0,1){25}}
 \put(30,-20){$\simeq$}
 \put(60,-20){(excision)}

 \put(-10,-40){$ H_{i+j-l}(|C|\cap |D|; |C|\cap |D|\cap J) $}
  \put(50,-70){\vector(0,1){25}}
 \put(30,-60){$\simeq$}
% \put(60,-20){(excision)}

\put(-10,-80){$ H_{i+j-l}(|C|\cap |D|;  (|\pa C|\cap |D|)\cup (|C|\cap |\pa D|)) $}
        \end{picture}
    \end{center}
\vskip 4cm

The last arrow is an isomorphism since the third term in the exact sequence of these pairs is isomorphic (by excision) to $ H_{i+j-l-1}(|C|\cap |D|\cap X_{sing};  ((|\pa C|\cap |D|)\cup (|C|\cap |\pa D|))\cap X_{sing})$ which is zero thanks to the allowability assumptions which imply that \begin{equation}\label{eq_last_arrow}
\dim |C|\cap |D|\cap X_{sing}\leq i+j-l-2.
\end{equation}

 The main property of this intersection product is that if $C\pitchfork D$, $C\pitchfork \pa D$ and $\pa C \pitchfork D$ we have:
\begin{equation}\label{eq_bord_pairing}
\pa (C \cap D) = \pa C  \cap D+ (- 1)^{l-i} C \cap \pa D.
\end{equation}
in $H_{i+j-l-1}((|\pa C|\cap |D|)\cup (|C|\cap |\pa D|) $.
This formula makes the pairings between the allowable cycles independent of the choice of the representatives of the classes, giving rise to a pairing between the homology groups \cite{gm1}.

\subsection{Intersection pairings on $\pa$-pseudomanifolds.}
Let now $X$ be an oriented stratified $\pa$-pseudomanifold and consider again three perversities $p$, $q$ and $r$ such that $p+q=r$. Recall that we defined allowable chains $C$ by requiring an allowability condition for $|C|\cap \pa X$ and  $cl(|C|\setminus \pa X)$. Hence,  it is natural to extend definition \ref{dfn_dim_trans} to $\pa$-pseudomanifolds by setting:

\begin{dfn}
 Two chains $C \in I^p C_i(X)$ and $D \in I^q C_j(X)$ {\bf are dimensionally transverse} if  for $2\leq m\leq l-1$: \begin{equation}\label{eq_allow_avec_bord}
\dim cl(|C|\cap |D|\cap  X_{\pa,reg})\cap X_{l-1-m}'\leq (i+j-l)-m+\check{r}_m .
\end{equation} and for $2\leq m\leq l$: \begin{equation}\label{eq_allow_avec_bord}
\dim cl(|C|\cap |D|\setminus  \pa X)\cap X_{l-m}\leq (i+j-l)-m+r_m .
\end{equation}
\end{dfn}

\bigskip

We wish to follow the same process as in \cite{gm1} to define  intersection pairings on $\pa$-pseudomanifolds. The only thing that we have to show is that the last arrow of the above diagram is an isomorphism.

 The problem is that on $\pa$-pseudomanifolds, if $C\in I^p C_i(X)$ and  $D \in I^q C_j (X)$ with $C\pitchfork D$, inequality  (\ref{eq_last_arrow}) may fail.
%$$\dim |C|\cap |D|\cap X_{sing}\geq i+j-l-1.
%$$
Nevertheless, it is possible to show the following lemma:
\begin{lem}If   $C\in I^p C_i(X)$ and  $D \in I^q C_j (X)$ are dimensionally transverse then  \begin{equation}\label{eq_lem_pairing}\dim cl(|C|\cap |D| \setminus X_{l-2})\cap X_{l-2} \leq i+j-l-2 .\end{equation}
\end{lem}
\begin{proof} Thanks to the allowability conditions, the desired inequality clearly holds if  $|C|\cap |D|$ is replaced by  $|C|\cap |D| \cap  \pa X$ or $|C|\cap |D|\setminus \pa X$ and thus it holds for $|C|\cap |D|$ itself as well.
%  As a matter of fact, it is enough to show:
%$$\dim cl(C \setminus X_{sing})\cap (|D|\cap \pa X)\cap X_{l-k} \leq i+j-l-2 .$$
%But
% We may focus on the former since for the latter the required inequality directly follows from the allowability  assumption (see (\ref{eq_allow_avec_bord})). The allowability assumption provides:
%$$ \dim C \cap X'_{l-1-k}\leq m -(k-1)+\check{r}_{k-1}\leq m-2 $$
%(where $X'_i:=X_i\cap \pa X$).  %Thus as $r_3=1$ we have  $\check{r}_{k-1}=r_k-1$ %and
%$$ \dim C\cap \pa X \cap X_{l-k}= \dim C \cap X'_{l-1-(k-1)}\leq m -k+r_k, $${\bf in the case where $(p_3+q_3)$ is $1$} (this is the case if $p$ and $q$ are complement perversities),
%which clearly implies the desired result.
\end{proof}

Let now $\xt$ be the double of $X$, i. e. the stratified pseudomanifold obtained by attaching two copies of $X$ along  $\pa X$.  As a consequence of the above  lemma, if $C \in I^p C_i(X)$ and $D \in I^q C_j(X)$ then
 the last arrow of the above diagram written for the pseudomanifold $\xt$  is an isomorphism since by  (\ref{eq_lem_pairing}), for each $a=0,1$, the boundary operator from $H_{i+j-l-a}(|C| \cap  |D| ;|C|\cap | D|\cap J)$ to $$H_{i+j-l-a-1}(|C| \cap  |D|\cap J;(|\pa C|\cap |D|)\cup (|C|\cap |\pa D|)\cap \xt_{sing})$$
 is then necessarily identically zero.  Denote by $\cap_{\xt}$ the resulting pairing.

\subsection*{Definition of the pairing.} There is a natural map $C_i(X)\to C_i(\xt)$, $C \to \tilde{C}$, assigning to every chain its double, mapping relative cycles of $(X;\pa X)$ into cycles of $\xt$ . Let $p$, $q$ and $r$ be three perversities with $p+q=r$. Given two chains $C \in I^p C_i  (X;\pa X)$ and $D \in I^q C_j  (X)$ such that $C\pitchfork  D$ we set:
$$C\cap D:=\tilde{C} \cap_{\xt} D. $$

\begin{lem}\label{lem_formule_pairing}
Let $C \in I^p C_i(X;\pa X)$  and $D \in I^q C_j  (X)$ with $C\pitchfork D$, $\pa C \pitchfork  D$ and $C\pitchfork \pa D$. We have:
\begin{equation}\label{eq_formule_pairing}
\pa (C\cap D)= \pa C \cap D+(-1)^{l-i} C\cap \pa D.
\end{equation}
%where $\pa_0 C$ is constituted by the simplices of $\pa C$ which do not belong to $C_{i-1}(\pa X)$.
\end{lem}
\begin{proof}
This formula is of course deduced from (\ref{eq_bord_pairing}) and the fact that $\pa \tilde{C}=\widetilde{\pa C}$.
\end{proof}

% Furthermore, formula (\ref{eq_bord_pairing}) still holds. We

%We are going to derive from the construction of \cite{gm1} an intersection pairing on $\pa$-pseudomanifold  between the homology groups $I^p H_j(X)$ and $I^q H_{l-j} (X)$. Let $X$ be a $\pa$-pseudomanifold and let $\tilde{X}$ be the double of $X$, i. e., the stratified pseudomanifold constituted  by attaching two copies of $X$ along $\pa X$.

  Obviously, the pairing is still well defined if one of the two chains is a Borel-Moore chain since supports of (finite) allowable chains are compact. We conclude:

\begin{pro}
Let $p$ and $q$ be complement perversities. For any $i$, then there is a unique intersection pairing
$$\cap \, :\, I^p H_i (X; \pa X)\times I^q H_{l-i}^{BM} (X)\to I^t H_{0}(X).$$
such that $[\sigma \cap  \tau] = [\sigma] \cap [\tau]$ for every dimensionally transverse pair of cycles.
\end{pro}
\begin{proof}
 This may be proved like in \cite{gm1}, replacing  (\ref{eq_bord_pairing})  by (\ref{eq_formule_pairing}).
\end{proof}

More generally, if $W$ denotes a subanalytic open subset of $X$, we have an  intersection  pairing:
$$I^p H_i (X;W\cup \pa X)\times I^q H_{l-i}^{BM} (X;W)\to I^l H_{0}(X).$$

\begin{rem}
We have derived our pairing from the one of \cite{gm1} by considering the double of $X$. One could also have considered relative forms of Lefschetz-Poincar\'e duality in the above diagram. It seems to lead  to the same pairing. The advantage of the method we used is that we avoided  reconsidering  the sequence of homomorphisms.
\end{rem}

\subsection{The Lefschetz duality morphism.}
If $X_{reg}$ is connected, then $I^l H_{0}(X)=\R$ and  this pairing gives rise to a homomorphism:
$$\psi_X^i: I^p H_i(X;\pa X) \to I^q H_{l-i} ^{BM} (X) ^*,$$
defined in the obvious way ($*$ denotes the dual functor i. e. $E^*= Hom (E;\R)$). We shall show that $\psi_X^i$ is an isomorphism for any $i$.

More generally, for any subanalytic open set $W \subset X$,  we have a map:
$$\psi_{X,W}^i: I^p H_i (X;W\cup \pa X) \to I^q H_{l-i}^{BM} (X;W) ^*.$$

\section{Lefschetz duality}\label{sect_lef}Let $X$ be a subanalytic oriented stratified $\pa$-pseudomanifold.
\begin{thm}\label{thm_lefschetz}
For any perversities $p$ and $q$ with $p+q=t$, the mappings $\psi_X^j$ induce isomorphisms:  $$I^{q}  H_{l-j}^{BM} (X)\simeq I^p H_{j} (X;\pa X).$$
In particular, if $X$ is compact:
 $$I^{q}  H_{l-j} (X)\simeq I^p H_{j} (X;\pa X).$$
\end{thm}
\begin{proof}
We prove the theorem by induction on $l=\dim X$. If $l=0$ the result is clear. Let $X \subset \R^n$ be a $\partial$-pseudomanifold and assume that the theorem holds true for any stratified  $\partial$-pseudomanifold.
%{\it First case: the case of products.}

Observe  that if $(X;\partial X)$ is a product $(Y\times (0;1);\partial Y \times (0;1))$ (with $(Y;\partial Y)$ stratified $\partial$-pseudomanifold  of $\R^{n-1}$) and is equipped with a product stratification then the result immediately follows from the induction hypothesis, together with  Lemma \ref{lem_produits} and Remark \ref{rem_produit}.

%Let $p'$ be the $(m-1)$ perversity defined by $p'_i:=p_i$ if $i \leq m-1$.  By Lemma \ref{lem_produits}:
%$$I^{p'}  H_j (Y)=I^p H_j (Y\times (0;1)).$$
%Thanks to the  induction on $n$, we know that the groups of the left-hand-side are finitely generated and independent of the chosen stratification. The result (in the case of products) follows.

In order to perform the induction step, we establish the following facts by downward induction on $m\leq l$:

 {\bf $(\textrm{A}_m)$.} The mappings $\psi_Y^j$ are isomorphisms for any set $Y$ of type $\{x \in X: \forall\, i\leq  m,\; |x_{i}-a_i| <\ep \}$, with $a_1,\dots,a_m$ real numbers, for $\ep>0$ small enough.

The  theorem follows from {\bf $(\textrm{A}_0)$}. We first prove  {\bf $(\textrm{A}_l)$}. We have to show that $X^\ep := B(x_0;\ep)\cap X$ satisfies Lefschetz duality.  If $x_0 \notin \partial X$, this follows from \cite{gm2} (see also \cite{k}).

It follows from the local conic structure of subanalytic sets that $X^\ep\setminus x_0$ is subanalytically homeomorphic to the product of the link by an interval, for which we already saw that Lefschetz duality holds.

 Hence, if $p_3=0$ and $q_3=1$ then,  by Lemmas \ref{lem_produits}  and \ref{lem_calcul local},  for $\ep$ small enough, Lefschetz duality holds for  $I^p H_j (X^\ep)$.

On the other hand, if $p_3=1$,  Lemmas \ref{lem_produits}  and \ref{lem_calcul local} also establish that $\psi_{X^\ep}^j$ is an isomorphism for $j\neq l-p_l-1$.  We are going to prove that  $\psi_{X^\ep}^{l-p_l-1}$ is an isomorphism as well.
  This is more delicate since in this case the inclusion of $X\setminus x_0$ does not induce  isomormophisms. Indeed, in this case Lemma \ref{lem_calcul_image} says that the only cycles which appear are those of the boundary.  Thus, we shall deduce the duality from the one of the boundary.

For simplicity let $j:=l-p_l-1$. We have the following commutative diagram:
\begin{center}
     \begin{picture}(-210,0)
\put(-120,3){$\pa $} \put(-120,-47){$i^*$}
\put(-200,-25){$\psi_{ X^\ep \setminus x_0}^{j}$}\put(-65,-25){$\psi_{\pa X^\ep \setminus x_0}^{j-1}$}
      \put(-280,-3){$I^p H_{j} (X^\ep \setminus x_0;\pa X^\ep \setminus x_0)$}
        \put(-90,-3){$I^p H_{j-1} ( \pa X^\ep \setminus x_0)$}
      \put(-250,-52){$I^q H_{l-j} ^{BM} (X^\ep \setminus x_0)^*$}
          \put(-90,-52){$I^q H_{l-j}^{BM} (\pa X^\ep \setminus x_0)^*$}
%      \put(-90,15){$ A_{1} $}
%      \put(-90,-45){$A_2$}
      \put(-160,-50){\vector(3,0){60}}
      \put(-160,0){\vector(3,0){60}}
  \put(-205,-10){\vector(0,-1){30}}
      \put(-70,-10){\vector(0,-1){30}}
     % \put(-130,-40){$h_2$}
%\put(-130,15){$h_1$}
     \end{picture}
    \end{center}
\vskip 20mm
where $i^*$ is induced by inclusion and $\pa$ by the boundary operator.
Since $\pa X^\ep$ is a pseudomanifold without boundary, $\psi_{\pa X^\ep \setminus x_0}^{j-1}$ is an isomophism and so  \begin{equation}\label{eq_ker}
\psi_{X^\ep \setminus x_0}^{j} (\ker \partial) =\ker i^*.
\end{equation}

Write now the following commutative diagram:
\begin{center}
     \begin{picture}(-210,0)
\put(-120,3){$b$} \put(-120,-47){$b'$}
\put(-160,-25){$\psi_{ X^\ep \setminus x_0}^{j}$}\put(-65,-25){$\psi_{X^\ep }^{j}$}
      \put(-280,-3){$I^p H_{j} (X^\ep \setminus x_0; \pa X^\ep \setminus x_0)$}
        \put(-90,-3){$I^p H_{j} ( X^\ep ;\pa X^\ep)$}
      \put(-250,-52){$I^q H_{l-j} ^{BM} (X^\ep \setminus x_0 )^*$}
          \put(-90,-52){$I^q H_{l-j}^{BM} ( X^\ep )^*$}
%      \put(-90,15){$ A_{1} $}
%      \put(-90,-45){$A_2$}
      \put(-140,-50){\vector(3,0){40}}
      \put(-140,0){\vector(3,0){40}}
  \put(-165,-10){\vector(0,-1){30}}
      \put(-70,-10){\vector(0,-1){30}}
     % \put(-130,-40){$h_2$}
%\put(-130,15){$h_1$}
     \end{picture}
    \end{center}
\vskip 20mm
where $b$ and $b'$ are induced by inclusion.  As  $X^\ep\setminus x_0$ is subanalytically homeomorphic to a product over the link, the first vertical arrow is an isomorphism.  We wish to show that so is the second vertical arrow. Indeed, by Lemma \ref{lem_local} the two above horizontal arrows are onto. Therefore is enough to show that $\psi_{X^\ep\setminus x_0}^j(\ker b )=\ker b'$. But, by Lemma \ref{lem_calcul_image} and (\ref{eq_ker})  we get:
$$\psi_{X^\ep \setminus x_0}^{j} (\ker b) =\psi_{X^\ep \setminus x_0}^{j} (\ker \pa)\overset{(\ref{eq_ker})}{=}\ker i^*=\ker b'.$$

 This yields {\bf $(\textrm{A}_l)$}.

Let $k<l$ and set $\pi_k(x)=x_k$. Let $Y$ be a set like in {\bf $(\textrm{A}_k)$}.  We shall write $Y_a$ for $Y \cap \pi_k ^{-1}(a)$ and $Y_{[a;b]}$ for $Y \cap \pi_k ^{-1}([a;b])$.

 By Hardt's theorem, there exists finitely many real numbers $-\infty =y_0,y_1,\dots,y_s,y_{s+1}=\infty$ such that on $(y_i;y_{i+1})$ the family $Y$ is topologically trivial.  We may assume that this trivialization preserves the strata and $\pa Y$.

Set $T_i:=Y_{(y_i-\ep;y_{i+1}+\ep)}$ and $Z_i:=Y_{(y_i+\frac{\ep}{2};y_{i+1}-\frac{\ep}{2})}$ as well as  $$W_i:=Y_{(y_i+\frac{\ep}{2};y_{i}+\ep)}\cup Y_{(y_i-\ep;y_{i}-\frac{\ep}{2})}$$ with $\ep>0$ small. Finally set $W :=\cup \, W_i$.
%For simplicity let

 Let us write the exact sequence (\ref{eq_pair_rel}) for the inclusion $(W;W\cap \pa Y) \hookrightarrow (Y;W\cup \pa Y)$:
$$\dots \to I^p H_j (W;W \cap \pa Y)\to I^p  H_j(Y;\pa Y)\to I^p  H_j(Y;W\cup \pa Y)\to \dots .$$
As $\pa W$ has a collared neighborhood (by topological triviality), by Lemma \ref{lem_exact_bm}, a  similar exact sequence holds for the dual groups of Borel-Moore intersection homology of the pair $(Y;W)$.  These two exact sequences constitute a commutative diagram with the mappings $\psi_Y^j$, $\psi_W^j$  and $\psi_{Y,W}^j$.

Therefore, thanks to the five lemma it is enough to show that the maps $\psi_W^j$'s and $\psi_{Y,W}^j$'s induce  isomorphisms on  $I^p H_j (W)$ and $ I^p H_j (Y;W)$.   By Hardt's theorem, $W$ is a product of a generic fiber  (which is a $\partial$-pseudomanifold) by an open interval. As we have established Lefschetz duality for products of pseudomanifolds of $\R^{n-1}$ by  an open interval, it remains to show that it is also true for $ I^p H_j (Y;W)$.

By excision:
$$ I^p H_j (Y;W) =\oplus _{i =1 } ^s  I^p H_j (T_i ;W_i)\oplus_{i=1} ^s I^p H_j (Z_i ; W_i). $$
Thus, it is enough  to  deal separately with $I^p H_j (T_i ;W_i)$ and $I^p H_j (Z_i ; W_i)$. Again, thanks to the exact sequences  of the pairs $(Z_i;W_i)$ and $(T_i;W_i)$ and the five lemma, it is enough to show Lefschetz duality for $W_i $, $T_i $ and $Z_i$. Thanks to topological triviality,  $W_i$ and $Z_i$ may be identified with  a product of the generic fiber by an open interval for which we observed that the result holds true.  For $T_i$, the result follows from {\bf $(\textrm{A}_{k+1})$}.
\end{proof}

%\subsection{Some computations of the $IH$-Euler characteristic.}
\subsection{Some concluding remarks and an example.}
\begin{enumerate}
\item The same inductive argument as in the above proof yields that the groups are finitely generated and independent of the chosen stratification.  Observe also that this Lefschetz duality result generalizes Theorem \ref{thm_poincare_ih}.
\item We may define $$I^p C_j^{BM}(X;\pa X):=\frac{I^p C_j^{BM}(X)}{I^p C_j^{BM}(\pa X)},$$
and denote by $I^p H_j^{BM}(X;\pa X)$ the resulting homology groups.  Then, we have a similar duality result:
$$I^p H_j^{BM}(X;\pa X)\simeq I^q H_{l-j}(X),$$
if $X$ is a subanalytic oriented stratified  $\pa$-pseudomanifold and $p+q=t$. This isomorphism may indeed be deduced from the one of the latter theorem and the five lemma  since we have a sequence for the Borel-Moore intersection  homology of the pair $(X;\pa X)$ which constitutes a commutative diagram with the one of singular intersection homology of the pair $(X;\pa X)$.
\item It seems that the results of this paper could be generalized to stratified  sets $X$ which are not stratified pseudomanifolds but which have a one codimensional stratum which is a stratified  pseudomanifold. However, the statement duality needs to be adapted. Nevertheless, the groups  seem to be an inveriant of $(X;X_{n-1})$. It would be interesting to compare this with the results obtained by Friedman in  \cite{f2}, where the author studied pseudomanifolds with possibly a one codimensional stratum and established theorems  of this type.

%Let $(X;A)$ be a pair stratified with a stratification such that $X\setminus A$  and $A$ are repsectiveluy $m$ and $(m-1)$ stratified pseudomanifolds. Then we may defined complexes $I^p _A C_j (X)$ in the same way (replacing $\pa X$ by $A$). The resulting groups seem to depend on $A$ but not on the stratification. It seems also that there is a duality result after a kind normalization of $X$: we may find $\hat{X} \to X$ (with $\hat{X}$ $\pa$ pseudomanifold)  making the  link at every stratum  connected

\item  Let $X$ be a pseudomanifold. In \cite{vlinfty}, it is proved that the $L^\infty$ cohomlogy of $X_{reg}$ is isomorphic to intersection cohomology  in the maximal perversity. This theorem is still true if $X$ is a $\pa$-pseudomanifold (the Poincar\'e Lemma proved in \cite{vlinfty} does not assume that $X$ is a pseudomanifold). In \cite{vl1}, we prove that the Dirichlet $L^1$ cohomology is always dual to $L^\infty$ cohomology. Therefore, the Lefschetz duality proved in the present paper implies that, if $X$ is a $\pa$-pseudomanifold,  the   Dirichlet $L^1$ cohomology of $X_{reg}$ is isomorphic to intersection cohomology of $(X;\pa X)$ in the zero perversity (compare with \cite{vl1} Corollary $1.6$).
\end{enumerate}

%\newpage

\begin{exa}
Consider the following double pinched torus embedded in $S^3=\R^3 \cup \infty$.

\begin{figure}[ht]
\begin{pspicture}(1,-2)(8,2)
\pscurve[linewidth=1pt](4,1.48)(1.5;0)(4,-1.48)
\pscurve[linewidth=1pt](4,1.48)(6.5;0)(4,-1.48)
\psellipse[linewidth=.8pt](4,0)(3.5,1.5)
\psellipse[linewidth=.8pt](1,0)(0.5,0.15) % \rput(5,-4){$a
%$}\rput(5,-3.7){\psline[linewidth=.8pt]{->}(0,0)(0,1)}
\end{pspicture}
\caption{}
\end{figure}

Consider the $\pa$-pseudomanifold $X$ consituted by this torus together with the connected component of its complement which is not simply connected (the unbounded one on the picture). The boundary of this $\pa$-pseudomanifold is this double pinched torus.

We first examine the intersection homology groups  for the top perversity  near a singular point $x_0$ of the singular torus. Let $X^\ep :=B(x_0;\ep)\cap X$. Then thanks to Lemma \ref{lem_calcul local} we get, $I^t H_2(X^\ep;\pa X^\ep)\simeq I^t H_0(X^\ep;\pa X^\ep)\simeq 0$ and:$$I^t H_1(X^\ep;\pa X^\ep)\simeq \R.$$

A representative of the generator of $I^t H_1(X^\ep;\pa X^\ep)$ is provided by any arc joining the two connected components of the regular locus of the torus. The groups $I^t H_j (X;\pa X)$ are indeed the same.

On the other hand, it is not difficult to show that  $I^0 H_1(X)\simeq I^0 H_3(X)\simeq 0$ and:
$$ I^0 H_2(X)\simeq \R.$$

The generator of $I^0 H_2(X)$ is given by any of the two cycles of the pinched torus.  We see in particular that this class does not have a $0$-allowable representative in $X\setminus \pa X$. The  $0$-allowability condition of chains in $\pa X$ (since $\check{0}=0$) is thus essential to ensure Lefschetz duality.
\end{exa}

\section{Appendix: two exact sequences} For the sake of clarity, we gather in this section a couple of exact sequences, derived in a fairly classical way and  needed in the proof of Lefschetz duality.
\subsection{Intersection homology relative to an open set.} Let $X$ be a subanalytic stratified $\partial$-pseudomanifold and  let $W$ be a subanalytic  open subset of $X$. We may endow this subset with the filtration $X_i \cap W$, where  $X_i$ denotes the given filtration of $X$.

  As a $p$-allowable chain of $W$ is obviously a $p$-allowable subset of $X$, we may set $$I^p H_j (X;W):= \frac{I^p H_j (X)}{I^p H_j (W)}.$$

 The inclusion $(W;W\cap \partial X)\hookrightarrow (X;\partial X)$ induces the following  long exact sequence:
\begin{equation}\label{eq_pair_rel}
\dots  \to I^{p} H_j (X;\partial X)\to I^p H_j (X;W \cup \partial X) \to I^p H_{j-1} (W;W \cap \partial X) \to \dots .
\end{equation}

\subsection{Borel-Moore intersection homology for $\partial$-pseudomanifolds.} %We define the Borel-Moore intersection homology in our setting in the way which is the most convenient for us. We compactify and take the intersection homology relatively to a neighborhood of infinity. This is in our setting of course consistent with  the  definition of Borel-Moore cohomology in terms of locally finite chains. We will need it to establish Lefschetz duality.
A similar exact sequence holds with the Borel-Moore homology. It is however somewhat more delicate and we have to assume that $W$ has a collared neighborhood.

Let $X$ be a subanalytic locally closed stratified pseudomanifold.

\begin{lem}\label{lem_isom_bm}
 Let $\hat{X}:=X\cup \{\infty \}$ be the one point compactification of $X$ (we can assume $\hat{X}$ subanalytic).  For $\ep$ small enough: $$I^p H_j ^{BM}(X)\simeq I^p H_j  (\hat{X};B(\infty;\ep)\cap \hat{X})= I^p H_j^{BM}  (\hat{X};B(\infty;\ep)\cap \hat{X}).$$
\end{lem}
\begin{proof}
Given $\sigma\in I^p C_j ^{BM}(X)$ we may assume, up to some locally finite  subdivision, that the support of the simplices of the chain  $\sigma$ which entirely lie in $U$ cover a neighborhood of $\infty$ in $\hat{X}$. This provides a map   $I^p C_j ^{BM}(X)\to\, \underset{\to}{\lim} \,I^p C_j (\hat{X};B(\infty;\ep)\cap \hat{X}).$ As $B(\infty;\ep)\cap \hat{X}$ is subanalytically homeomorphic to a cone of $S(\infty;\ep)\cap \hat{X}$, it is not difficult to show that this morphism gives an isomorphism in homology.
\end{proof}

\subsection*{The exact sequence of a pair.}
    Let $W$ be an open subanalytic subset of $X$.

\begin{lem}\label{lem_exact_bm}
 If $\partial W:= X\cap cl(W)\setminus W $ has  a stratified collared neighborhood in $cl(W)\cap X$, we have an exact sequence:
$$\dots \to I^p H_j ^{BM}(X;W) \to I^p H_j ^{BM}(X) \to I^p H_{j} ^{BM} (W) \to   I^p H_{j-1} ^{BM} (X;W) \to \dots . $$
\end{lem}
\begin{proof}
Given an open set $V$ of $X$ let:
$$\hat{I}^p C_j ^{BM}(X;V):= \frac{I^p C_j^{BM} (X)}{I^p C_j^{BM} (X;V)},$$ and denote by $\hat{I}^p H_j ^{BM}(X;V)$ the resulting homology groups.

Since $\partial W$ has a collared neighborhood, there is a stratified subanalytic  mapping $h: \partial W \times [0;1) \to W$. Let $W_t:=W \setminus h(\partial W \times (0;t])$.

 By the preceding lemma, as the family $W \setminus W_t$ constitutes a fundamental system of neighborhoods of $\pa W$, we have, thanks to the preceding lemma:  \begin{equation}\label{eq_W_t}
\hat{I}^p H_j (W; W_t) =I^p H_j^{BM} (W)\end{equation}

It follows from an excision argument that:
\begin{equation}\label{eq_isom1}\hat{I}^pH_j^{BM}(X;W_t)\simeq \hat{I}^p H_j^{BM} (W;W_t)\overset{\mbox{(\ref{eq_W_t})}}{\simeq} I^p H_j^{BM} (W).\end{equation}

As $h$ is a homeomorphism, the inclusion $I^p C_j ^{BM} (X;W_t) \hookrightarrow I^p C_j ^{BM} (X;W)$ induces an isomorphism between the homology groups:\begin{equation}\label{eq_isom_incl}
I^p H_j ^{BM} (X;W) \simeq I^p H_j ^{BM} (X;W_t).
\end{equation}

As usual, the short exact sequences
$$0\to I^p C_j ^{BM}(X; W_t)\to  I^p C_j ^{BM} (X)\to\hat{I}^p C_j ^{BM}(X;W_t)\to 0,$$
give rise to a long a exact sequence. By (\ref{eq_isom1}) and (\ref{eq_isom_incl}) this exact sequence is the desired one. %  Hence by (\ref{eq_isom_incl}), it is enough to show that  \begin{equation}\label{eq_isom2}
% I^p H_j ^{BM}(W)\simeq \hat{I}^p H_j ^{BM}(X;W_t) .\end{equation}
 %As, all the simplices of $X \setminus  W$ are zero in $\hat{I}^p C_j ^{BM}(X;W_t)$ the inclusions
 %$$\hat{I}^p C_j^{BM} (W;W_t) \hookrightarrow \hat{I}^p C_j ^{BM}(X;W_t) ,$$
%induce  isomorphisms between the homology groups. Now, as $U\supset W \setminus W_t$, by Lemma \ref{lem_exact_bm} we can derive from this isomorphism:
 %$$I^p H_j^{BM}(W )\simeq  I^p H_j^{BM} (W;W_t) \simeq  \hat{I}^p H_j ^{BM}(X; W_t),$$
%establishing (\ref{eq_isom2}).
\end{proof}
%\medskip

In the  exact sequence of the above lemma,  the boundary operator  coincides with the boundary operator on generic representative of chains.
%Furthermore:
%$$I^{p'} H ^{BM}_{j} (X)=I^p H_{j+1} ^{BM} (X\times (0;1)),$$
%and the same holds true for the pair $(X;\partial X).$
 %For the Borel-Moore homology, observe that if $\sigma$ is a cycle of $ I^{p'} H ^{BM}_{j} (X)$ then $\sigma \times [\ep,1-\ep]$ provides for $\ep>0$ small enough  a Borel-Moore cycle of $X \times (0;1)$.


\begin{thebibliography}{mmm}
\bibitem[F1]{f1}G. Friedman, Superperverse intersection cohomology: stratification (in)dependence,
Mathematische Zeitschrift 252 (2006), 49 - 70.
\bibitem[F2]{f2}G. Friedman,  An introduction to intersection homology with general perversity functions,
to appear in the Proceedings of the Workshop on the Topology of Stratified Spaces at MSRI, September 8-12, 2008.
\bibitem[F3]{f3}G. Friedman,
Intersection homology with general perversities,
Geometria Dedicata 148 (2010), 103-135.
\bibitem[GM1]{gm1}M. Goresky, R.   MacPherson, Intersection homology theory,  Topology  19  (1980), no. 2, 135--162.
\bibitem[GM2]{gm2}M. Goresky, R.   MacPherson, Intersection homology II.  Invent. Math.  72  (1983), no. 1, 77--129.
\bibitem[K]{k} F. Kirwan, An introduction to intersection homology theory, Chapman HallCRC (2006), 248 pages.
\bibitem[V1]{vlinfty} G. Valette, $L^\infty$ cohomology is intersection conomology, prerpint, 	arXiv:0912.0713v1.
\bibitem[V2]{vl1} G. Valette, $L^1$ cohomology of bounded subanalytic manifolds, prerpint,	arXiv:1011.2023v1.
\end{thebibliography}
\end{document}